\documentclass[12pt]{article}

\usepackage{amssymb,amsmath,amsthm,mathrsfs,mathtools}

\usepackage[affil-it]{authblk}

\usepackage[left=1in, right=1in]{geometry}
\usepackage{appendix}
\usepackage[round]{natbib}
\newcommand{\eps}{\varepsilon}

\def\R{{\mathbb R}}

\def\t{{\mathbb T}}

\def\d{{\,\rm d}}

\def\w{{\rm w}}

\def\[{\left\lfloor}
\def\]{\right\rceil}
\def\S{{\mathcal S}}
\def\:{\colon}

\def\be#1{\begin{equation}\label{#1}}
\def\ee{\end{equation}}

\newtheorem{theorem}{Theorem}
\newtheorem{lemma}[theorem]{Lemma}

\newtheorem{corollary}[theorem]{Corollary}

\newtheorem{definition}[theorem]{Definition}

\begin{document}

\title{Energy conservation for the Euler equations on $\t^2\times \R_+$ for weak solutions defined without reference to the pressure}
\author{James C.\ Robinson}
\affil{Mathematics Institute, University of Warwick, Coventry, CV4 7AL. UK.}
\author{Jos\'e L.\ Rodrigo}
\affil{Mathematics Institute, University of Warwick, Coventry, CV4 7AL. UK.}
\author{Jack W.D.\ Skipper}
\affil{Institute of Applied Mathematics, Leibniz University Hannover,
 Welfengarten~1, 30167 Hannover. Germany.}
\maketitle
\begin{abstract}
 We study weak solutions of the incompressible Euler equations on $\t^2\times \R_+$; we use test functions that are divergence free and have zero normal component, thereby obtaining a definition that does not involve the pressure. We prove energy conservation under the
assumptions that $u\in L^3(0,T;L^3(\t^2\times \R_+))$,
$$
 \lim_{|y|\to 0}\frac{1}{|y|}\int^T_0\int_{\t^2}\int^\infty_{x_3>|y|} |u(x+y)-u(x)|^3\d x\d t=0,
$$
and an additional continuity condition near the boundary: for some $\delta>0$ we require $u\in L^3(0,T;C^0(\t^2\times [0,\delta])))$. We note that all our conditions are satisfied whenever $u(x,t)\in C^\alpha$, for some $\alpha>1/3$, with H\"older constant $C(x,t)\in L^3(\t^2\times\R^+\times(0,T))$.
\end{abstract}

\section{Introduction}

Energy conservation for solutions of the incompressible Euler equations
$$
 \partial_t u+ (u\cdot\nabla)u +\nabla p=0 \quad \nabla \cdot u=0
$$
on domains without a boundary ($\R^d$ or $\t^d$ with $d\ge 2$) is now well understood. This problem has been studied extensively by \cite{constantin1994onsager}, \cite{duchon2000inertial},  \cite{cheskidov2008energy}, \cite{shvydkoy2010lectures} (see also \citealp{RRS1}) who have all proved energy conservation with varying  conditions on the solution. These conditions are all weaker than $u\in C^{1/3+\eps}$ for some $\eps>0$ and thus any solution satisfying $u \in C^{1/3+\eps}$ will conserve energy, that is, $\|u(t)\|_{L^2}=\|u(0)\|_{L^2}$ for every $t\ge0$.

These results prove the `positive' part of the `Onsager Conjecture' \citep{onsager1949statistical}: 
solutions with spatial regularity $C^{1/3+\eps}$ will conserve energy.  Recently \cite{isett2016} and \cite{delellis2017} have constructed solutions with regularity $C^{1/3-\eps}$ that do not conserve energy (in fact they show the existence of solutions that satisfy any prescribed energy profile).

In the case with boundary, it is easy to see, using standard integration-by parts techniques, that energy is conserved for a $C^1$ solution on a Lipschitz domain $\Omega$ with the solution $u$ satisfying $u\cdot n=0$ on $\partial\Omega$. In \cite{RRS1} we obtained sufficient conditions, similar to those presented here, for energy conservation in $\t^2 \times \R_+$, using a weak formulation that required a pressure term on the boundary. However, in our subsequent analysis the pressure played a very minimal role.

\cite{bardos2017} have shown energy conservation for $C^2$ bounded domains under the assumption $u \in L^3((0,T);C^{0,\alpha}(\bar{\Omega}))$ for $\alpha>1/3$; their definition of a weak solution requires a pressure function defined throughout the domain, and their result requires a careful analysis of this pressure.

In this paper we consider a solution $u$ on the spatial domain $\t^{2} \times \R_+$ and present an approach that  completely avoids the use of the pressure. It also involves conditions that are less restrictive that the $C^{1/3+\varepsilon}$ result of Bardos and Titi.
More precisely, we will show that for a solution $u$ to conserve energy it suffices that
$u\in L^3(0,T;L^3(\t^2\times \R_+))$ and
\begin{equation}\label{bulkintro}
 \lim_{|y|\to 0}\frac{1}{|y|}\int^T_0\int_{\t^2}\int^\infty_{x>|y|} |u(x+y)-u(x)|^3\d x\d t=0,
\end{equation}
along with a continuity condition near the boundary: $u\in L^3(0,T;C^0(\t^2\times [0,\delta]))$ for some $\delta>0$.
The bulk condition in \eqref{bulkintro} is very similar to the best known condition for the spatial domains $\R^d$ or $\t^d$, the only difference being that the domain of integration restricts to the interior of the domain.


\medskip

The plan of the paper is as follows. Section 2 contains some preliminary material and our definition of a `weak solution' of the Euler equations. In Section 3 we introduce a reflection and extension map to the full domain. In Section 4 we show that it is possible to test the weak formulation of the Euler equation with a mollification of the extended solution constructed in the previous section. Section 5 contains the main statement and its proof.

\section{Weak solutions of the Euler equations on $\t^2\times \R_+$}

In this section we introduce some basic notation and make precise the notion of weak solution of the Euler equation that we will be using.

For vector-valued functions $f,g$ and matrix-valued functions $F,G$ we use the notation
\be{ipd}
 \langle f,g \rangle_{\Omega}=\int_{\Omega}f_i(x)g_i(x) \d x \quad \mathrm{and}  \quad \langle F:G \rangle_{\Omega}=\int_{\Omega}F_{ij}(x)G_{ij}(x) \d x
\ee
using Einstein's summation convention (sum over repeated indices).

We let $\t^2$ denote the two-torus, write $\R_+$ for $[0,\infty)$, and define
$D_+:=\t^2\times \R_+$. We use the notation $\mathcal{S}(D_+\times [0,T])$ to denote functions in $C^\infty(D_+\times [0,T])$ that have Schwartz-like decay in the unbounded spatial direction, i.e.
\begin{equation}\label{Schwartzx3def}
 \sup_{(x,t)\in D_+\times [0,T]}|\partial^\alpha \phi||x_3|^\beta < \infty,
\end{equation}
for all integers $\beta\ge 0$ and all nonnegative multi-indices $\alpha$  over the variables $(x_1,x_2,x_3,t)$. Similarly, when there is no time component, the notation $\mathcal{S}(D_+)$ denotes functions in $C^\infty(D_+)$ that have Schwartz-like decay in the unbounded spatial direction as in \eqref{Schwartzx3def}.



We set
$$
\mathcal{S}_{n,\sigma}(D_+):= \{\phi\in \mathcal{S}(D_+):\text{div } \phi =0 \text{ and } \phi\cdot n = 0 \text{ on } \partial D_+\}
$$
and define the space $H_\sigma(D_+)$ as
$$
 H_\sigma(D_+):= \text{the completion of }  \mathcal{S}_{n,\sigma}(D_+) \text{ in the } L^2(D_+)\text{ norm}.
$$
Functions in $H_\sigma(D_+)$ are weakly divergence free in that they satisfy
\begin{equation}\label{Hincompressibility}
 \langle u,\nabla \phi \rangle_{D_+}=0\quad \text{for every}\quad  \phi \in H^1(D_+).
\end{equation}
This  holds  since $\S_{n,\sigma}(D_+)$ is dense in $H_\sigma(D_+)$,  and so for any $u\in H_{\sigma}(D_+)$ we can find $(u_n)\in\S_\sigma(D_+)$ such that $u_n\to u$ in $H^1(D_+)$. Now given $u\in H_{\sigma}(D_+)$  and any $\phi\in H^1(D_+)$ we have
 $$
 \langle u,\nabla\phi\rangle_{D_+}=\lim_{n\to\infty}\langle u_n,\nabla\phi\rangle_{D_+}=\lim_{n\to\infty}\langle\nabla\cdot u_n,\phi\rangle_{D_+}=0.
 $$
Notice that we have no boundary terms in the integration-by-parts since $u_n\cdot n=0$ on $\partial D_+$ (see for example Lemma 2.11 in Robinson et al., 2016, for more details).
%

In a slight abuse of notation we define $C_\w([0,T];H_\sigma (D_+))$ as the collection of all functions $u:[0,T]\to H_\sigma(D_+)$ that are weakly continuous into $L^2(D_+)$ i.e.
$$
 t\mapsto \langle u(t), \phi \rangle_{D_+}
$$
 is continuous for every $\phi\in L^2(D_+)$.


We define the space of test functions
\begin{equation}
\mathcal{S}_{n,\sigma}(D_+\times [0,T]):= \{\psi\in \mathcal{S}(D _+\times [0,T])\colon \nabla \cdot \psi(\cdot,t)=0\quad \psi\cdot n=0\; \mathrm{ on }\;  \partial D_+\quad\forall t\in [0,T] \}.
\end{equation}

Analogous definitions of all of the spaces above can be made for  the domain $D_-:= \t^2 \times \R_-$ (where $\R_-=(-\infty,0]$).

To obtain a weak formulation on $D_+$ assume that we have a smooth solution $u$ with pressure $p$ that satisfy the incompressible Euler equations
$$
 \begin{cases}
  \partial_t u+\nabla \cdot (u\otimes u) +\nabla p=0 &\text{in} \, D_+\\
  \nabla \cdot u=0 &\text{in} \, D_+\\
u \cdot n=0 &\text{on} \, \partial D_+,
 \end{cases}
$$
where $n$ is the outer normal to $\partial D_+$, so that for our domain the third equation simply becomes $u_3=0$ on $\partial D_+$. Taking the inner product of  the first equation with a vector-valued test function $\phi \in \mathcal{S}_{n,\sigma}(D_+\times [0,T]) $  and integrating over the time interval $(0,t)$ we obtain
$$
 \int^t_0 \langle \partial_t u+\nabla \cdot (u\otimes u) +\nabla p ,\phi \rangle_{D_+} \d \tau =0.
$$
Here $\langle \cdot,\cdot \rangle_{D_+}$ denotes the $L^2$-inner product in space as defined in \eqref{ipd}.
We can now integrate by parts and obtain
\begin{multline*}
 \langle u(t),\phi(t) \rangle_{D_+}-\langle u(0),\phi(0) \rangle_{D_+}- \int^t_0 \langle  u ,\partial_t\phi \rangle_{D_+}\d \tau  - \int^t_0 \langle  (u\otimes u):\nabla \phi \rangle_{D_+} \d \tau\\- \int_{\partial D_+ \times [0,t]} u_3 \,( u\cdot \phi) \d S_x \d t- \int ^t_0 \langle p,\nabla \cdot \phi \rangle_{D_+} \d \tau +\int_{\partial D_+\times [0,t]} p \,\phi_3\d S_x \d t =0.
\end{multline*}
We notice that both $u_3=0$ and $\phi_3=0$  on $\partial D_+$. Further, we have that $\nabla \cdot \phi=0$ in $D_+$ and so the three terms involving these expression vanish; we obtain the equation
$$
 \langle u(t),\phi(t) \rangle_{D_+}-\langle u(0),\phi(0) \rangle_{D_+}- \int^t_0 \langle  u ,\partial_t\phi \rangle_{D_+}\d \tau  - \int^t_0 \langle  (u\otimes u):\nabla \phi \rangle_{D_+} \d \tau=0.
$$
Thus we have obtained the following weak formulation of the equation, which does not involve any pressure terms.

\begin{definition}[Weak Solution on $D_+$]\label{defEulersoln}
A weak solution of the Euler equations on $ D_+ \times [0,T]$  is a  vector-valued function   $u$ in  $C_\w([0,T];H_\sigma (D_+))$
such that
\begin{multline}\label{WeaksolutionD+}
 \langle u(t),\psi(t)\rangle_{D_+}-\langle u(0),\psi(0)\rangle_{D_+}-\int^t_0\langle u(\tau),\partial_t\psi(\tau)\rangle_{D_+} \d \tau\\=\int^t_0\langle u(\tau)\otimes u(\tau):\nabla\psi(\tau)\rangle_{D_+}\d \tau,
\end{multline}
for every $t\in[0,T]$ and for all $\psi\in \mathcal{S}_{n,\sigma}(D_+ \times [0,T])$.
\end{definition}



We conclude this section making precise the specific mollification  that we will use to regularise the equation. Throughout the paper $\varphi$ will be a radially symmetric scalar function in $C_c^\infty((-\frac{1}{2},\frac{1}{2})^3)$  with $\int \varphi=1$; we set $\varphi_\eps(x)=\eps^{-3}\varphi(x/\eps)$. Then for any function $f$ we define the mollification of $f$ as $J_\eps f:= f \star \varphi_\eps$ where $\star$ denotes convolution. Thus
\begin{equation}\label{mol}
J_\eps f(x)= f \star \varphi_\eps(x):=\int_{D}\varphi_\eps(x-y)f(y)\d y= \int_{B(0,\eps)}\varphi_\eps(y)f(x-y)\d y.
\end{equation}
\noindent Notice that given the way we have defined our mollification we need the  functions to be defined on all of $D:= \t^2 \times \R$. When applying this mollification to functions only defined on  $D_+$ we will implicitly assume an extension by zero to the entirety of $D$ prior to mollifying.

\section{The reflection map}

The first step in our analysis will generate an extension of a weak solution $u$ defined in $D_+$ to a function $u_E$ defined on all of $D$. We remark that we are using the same extension considered in \cite{RRS1}. In that work part of the considerations related to this particular extension were used to handle the pressure, which is not present in our current approach.

The extension will be built out of an {\it odd} reflection  $u$ from $D_+$ to $D_-$. However, for later convenience we consider a reflection map for functions defined in the full domain $D$; we will apply this later to an extension by zero of functions defined on the half space $D_+$. 

\begin{definition}[Reflection and extension]
 Given a vector-valued function $f\colon D\to \R^3$  we define $f_R\colon D\to \R^3$ by
\begin{equation}\label{reflection}
 f_R(x,y,z):= \begin{pmatrix}
       f_1(x,y,-z)\\
          f_2(x,y,-z)\\
          -f_3(x,y,-z)\\
         \end{pmatrix}.
\end{equation}
 For a function $g\colon D_+\to \R^3 $, defined only on $D_+$ to start with, we first consider a trivial extension by zero, which by an abuse of notation we still denote by g, and define $g_R$ via \eqref{reflection}. We now define our extension $g_E$  by
\begin{equation}\label{extension}
 g_E(x,y,z):=\begin{cases}
  g(x,y,z) +g_R(x,y,z) & z\neq 0\\
  \tfrac{1}{2}(g(x,y,z) +g_R(x,y,z))=(g_1(x,y,0),g_2(x,y,0),0) & z=0.
 \end{cases}
\end{equation}


\end{definition}
In \eqref{extension} we require a separate definition for $z=0$ to preserve the value of $g$ at the boundary of $D_+$, but we still have $g_E$  equal to $g+g_R$ almost everywhere.


 Note that if $g\in \mathcal{S}_{n,\sigma}(D_+)$ then $ g_R\in \mathcal{S}_{n,\sigma}(D_-)$; similarly if $v \in  H_\sigma(D_+)$ then we have $v_R \in H_\sigma(D_-)$.

We have defined this particular extension to preserve the function's incompressibility, regularity and boundary conditions. Additionally, we chosen the mollifying kernel in \eqref{mol} so that the mollification of $v_E$ satisfies all the properties of a test function for the equation. This will allow us to use it to regularise the  equation and manipulate the terms. We summarise some of the results we will require.

\begin{lemma}\label{PropertiesofE}
If $v\in H_\sigma(D_+)$ (respectively $\mathcal{S}_{n,\sigma}(D_+)$)  then $v_E\in H_\sigma(D) $ (respectively $\mathcal{S}_{\sigma}(D)$) and  
\begin{enumerate}
	\item $\|v_E\|_{L^p(D)}\le C\|v\|_{L^p(D_+)}$; 
	\item $J_\eps(v_E)$ and $J_\eps(J_\eps(v_E))$ are incompressible in $D_+$; and
	\item  $J_\eps(v_E)\cdot n=0$ and $J_\eps(J_\eps(v_E))\cdot n=0$ on $\partial D_+$. 
\end{enumerate}
\end{lemma}
\begin{proof}
We consider only the case $v\in H_\sigma(D_+)$. Given the initial extension of $g$ by zero, and that as remarked before $v_R \in H_\sigma(D_-)$, we only need to show that $v_E$ remains weakly incompressible. 
Since $v\in H_\sigma(D_+)$ there exists $v_n\in \S_{n,\sigma}(D_+)$ such that $v_n\to v $ in $L^2(D_+)$. Clearly $v_{n,R}\in \S_{n,\sigma}(D_-)$ and $v_{n,R}\to v_R$ in $L^2(D_-)$. Therefore $v_R\in H_\sigma(D_-)$. Further, $v_n+v_{n,R}$ trivially belongs to $S_\sigma (D)$ and is divergence free. Since $v_n+v_{n,R}$ converges to $v_E$ in $L^2(D)$ we obtain the desired result

Estimate {\it 1} follows easily
\begin{equation}
\|v_E\|_{L^p(D)} =  \|v +v_R\|_{L^p(D)}\le \|v\|_{L^p(D_+)} +\|v_R\|_{L^p(D_-)} \le 2 \|v\|_{L^p(D_+)}
\end{equation}
as $ \|v_R\|_{L^p(D_-)}=\|v\|_{L^p(D_+)}$.

In order to prove {\it 2}, since the extension is weakly incompressible we have that  $J_\eps(u_E)$
is strongly incompressible. To show this note that $v_E\in H_\sigma(D)$ and so $\langle v_E, \nabla \phi \rangle_D=0$ for all $\phi \in \mathcal{S}(D)$. We can let $\phi = J_\eps \eta \text{ or } J_\eps J_\eps \eta $ and thus
$$
	 0=\langle v_E, \nabla J_\eps \eta \rangle_D= \langle  J_\eps v_E, \nabla  \eta \rangle_D=\langle  \nabla \cdot  J_\eps v_E,  \eta \rangle_D.
$$
As this holds for all $\eta \in \mathcal{S}(D)$ we have that $ J_\eps v_E$
is strongly incompressible in $D_+$. We argue similarly for $J_\eps J_\eps v_E$.

To show {\it 3} we will first show that $J_\eps(v_E)_3=0$ on $\partial D_+$. Note that this is the same as $J_\eps((v_E)_3)=0$. As our extension is an odd function in the third component and $\varphi_\eps$ is an even function in the third component we have that the integral over the ball centered around the boundary is zero. We argue similarly for $J_\eps J_\eps v_E$.
\end{proof}

We now define
$$
D_{>s}:=\{x\in D : x_3> s\}.
$$
Notice that estimate {\it 1} in the previous Lemma holds for these domains as well. In fact, for any $\delta>0$ we have
$$
\|u_E\|_{L^p(D_{>-\delta})}\le \|u_E\|_{L^p(D)} \le  C\|u\|_{L^p(D_+)}.
$$

%

\begin{lemma}\label{Mollifierconvegence}
Given $u \in L^p(D_+)$ with $1\le p<\infty$ we have  $\|J_\eps(u_E)-u\|_{L^p(D_+)} \to  0$ and
 $\|J_\eps J_\eps(u_E)-u\|_{L^p(D_+)} \to  0$.
\end{lemma}

\begin{proof}
 The result follows easily by noticing that in $D_+$ we have $u_E=u$ and therefore
$$
  \|J_\eps (u_E)-u\|_{L^p(D_+)}=\| J_\eps(u_E)-u_E\|_{L^p(D_+)} \leq \| J_\eps(u_E)-u_E\|_{L^p(D)}.
$$
The result now follows by standard properties of mollifiers. Similarly for $J_\eps J_\eps u_E$.
\end{proof}

We conclude this section with a lemma collecting various results for the reflection map that will be used later.

\begin{lemma}\label{Rsymmetry}
 For any functions $u$ and $v$ on $D$
$$
  \langle u,v_R\rangle_{\t^2\times (-\delta,\delta)}=\langle u_R,v \rangle_{\t^2\times (-\delta,\delta)}
$$
for any $\delta>0$.
In addition
$$
  J_\eps (f_R)(x)=(J_\eps f)_R(x)
$$
and thus
$$
  \langle J_\eps u, J_\eps v_R\rangle_{\t^2\times (-\delta,\delta)}=\langle J_\eps u_R,J_\eps v \rangle_{\t^2\times (-\delta,\delta)}.
$$
\end{lemma}

\begin{proof}
The first part follows by a simple change of variables of $x_3$ to $-x_3$, using the symmetry of the domain of integration and the definition of the reflection map.
More precisely, using the notation $x=(\tilde x,x_3)$ we can use the change of variables $x_3=-\xi_3$ so that
\begin{align*}
  \langle u,v_R\rangle_{\t^2\times(-\delta,\delta)}&= \int_{\t^2} \int^\delta_{-\delta}u_i(\tilde x,x_3) v_{Ri}(\tilde x,x_3)\d x_3 \d \tilde x =
 \int_{\t^2} \int^\delta_{-\delta}u_i(\tilde x,-\xi_3) {v_{R}}_i(\tilde x,-\xi_3)\d \xi_3 \d \tilde x \\
&=\int_{\t^2} \int^\delta_{-\delta}{u_{R}}_i(\tilde x,\xi_3) v_{i}(\tilde x,\xi_3)\d \xi_3 \d \tilde x = \langle u_R,v\rangle_{\t^2\times(-\delta,\delta)}.
\end{align*}

The result  $J_\eps (f_R)=(J_\eps f)_R$ follows by a direct calculation (given the properties of our mollifying kernel), and for the final equality we apply the first part to $J_\eps u$ and $J_\eps v$.
%
\end{proof}

\section{Using $J_\eps J_\eps u_E$ as a test function}\label{WSM}

We will show that if $u$ is a weak solution then in fact \eqref{WeaksolutionD+} holds for a larger class of test functions with less time regularity. We denote by $C^{0,1}([0,T];H_\sigma)$ the space of Lipschitz functions from $[0,T]$ into $H_\sigma$. Here we extend the results in \cite{RRS1}, highlighting only the changes and generalisations needed because of the boundary.

\begin{lemma}\label{generalpsi}
	If $u$ is a weak solution of the Euler equations on $D_+$ then \eqref{WeaksolutionD+} holds for every $\psi\in\mathcal{L}_{n,\sigma}$, where
	\begin{equation*}
	\mathcal{L}_{n,\sigma}:=\{\psi\in L^1(0,T;H^3)\cap C^{0,1}([0,T];H_\sigma): \psi \cdot n=0 \;\mathrm{ on }\; \partial D_+\}.
	\end{equation*}
\end{lemma}


 \begin{proof} For a fixed $u$ we can write \eqref{WeaksolutionD+} as $E(\psi)=0$ for every $\psi\in \mathcal{S}_{n,\sigma}$, where
 	\begin{align*}
 	E(\psi):=\langle u(t),\psi(t)\rangle_{D_+}-\langle u(0),\psi(0)\rangle_{D_+}-\int^t_0\langle u(\tau),&\partial_t\psi(\tau)\rangle_{D_+} \d \tau\\
 	&-\int^t_0\langle u(\tau)\otimes u(\tau):\nabla\psi(\tau)\rangle_{D_+}\d \tau.
 	\end{align*}
 	Since $E$ is linear in $\psi$, and  $\mathcal{S}_{n,\sigma}$ is dense in $\mathcal{L}_{n,\sigma}$ with respect to the norm
 	\begin{equation*}
 	\|\psi\|_{L^1(0,T;H^3)}+\|\psi\|_{C^{0,1}([0,T];L^2)},
 	\end{equation*}
 	to complete the proof it suffices to show that $\psi\mapsto E(\psi)$ is bounded in this norm. We proceed term-by-term:
 	\begin{align*}
 	\left|\langle u(t),\psi(t)\rangle_{D_+}-\langle u(0),\psi(0)\rangle_{D_+}\right|&\le 2\|u\|_{L^\infty(0,T;L^2)}\|\psi\|_{L^\infty(0,T;L^{2})},\\
 	\left|\int^t_0\langle u(\tau),\partial_\tau\psi(\tau)\rangle_{D_+} \d \tau\right|&\le \|u\|_{L^\infty(0,T;L^2)}\|\psi\|_{C^{0,1}([0,T];L^2)},\\
 	\left|\int^t_0\langle u(\tau)\otimes u(\tau):\nabla\psi(\tau)\rangle_{D_+}\d \tau\right|&\le \|u\|^2_{L^\infty(0,T;L^2)}\|\nabla\psi\|_{L^{1}(0,T;L^{\infty})}.
 	\end{align*}
 (For details of the second of these estimates see \citealp{Skipper} or \citealp{RRS1}.)  	It follows that
 	\begin{equation*}
 	|E(\psi)|\le C \|u\|_{L^\infty(0,T;L^2)}\|\psi\|_{C^{0,1}([0,T];L^{2})} +C\|u\|^2_{L^\infty(0,T;L^2)}\|\psi\|_{L^{1}(0,T;H^3)}
 	\end{equation*}
 	and so we obtain the desired result. Note that $\psi\cdot n=0$ is preserved as $H^3\subset C^0$ in three dimensions.
 \end{proof}

We now study the time regularity of $u$ when paired with a sufficiently smooth function that is not necessarily divergence free.

\begin{lemma}\label{uLipcitz}
	If $u$ is a weak solution on $D_+$ then
	\begin{equation}\label{Lipschitzintime}
	|\langle u(t)-u(s),\psi\rangle_{D_+}|\le C|t-s|\quad \text{for} \, \text{all} \quad \psi\in \mathcal{S}(D_+),
	\end{equation}
	where $C$ depends only on $\|u\|_{L^\infty(0,T;L^2)}$ and $\|\psi\|_{H^3}$. Further, we have
	\begin{equation}\label{LipschitzintimeD}
	|\langle u(t)-u(s),\psi\rangle_{D}|\le C|t-s|\quad \text{for} \, \text{all} \quad \psi\in \mathcal{S}(D).
	\end{equation}
\end{lemma}


We remark that inequality \eqref{Lipschitzintime} holds for $\psi\in H^3(D_+)$, while \eqref{LipschitzintimeD} holds for $\psi\in H^3(D)$ as those are the norms of $\psi$ that appear in $C$. Therefore we can use density to extend the lemma to these larger spaces of functions.

 \begin{proof}
 	First observe that
 	any $\psi\in \mathcal{S}(D_+)$ can be decomposed as
 	\begin{equation*}
 	\psi=\eta+\nabla \sigma,
 	\end{equation*}
 	where $\eta,\sigma\in \mathcal{S}(D_+)$ and $\eta$ is divergence free with $\eta\cdot n=0$ on $\partial D_+$ (see Theorem 2.16 in Chapter 2 of \cite{RRS}, for example). Furthermore we have the bound
 	\begin{equation*}
 	\|\nabla \eta\|_{L^\infty}\le \|\nabla \eta \|_{H^2} \le \|\eta\|_{H^3} \le C\|\psi\|_{H^3}.
 	\end{equation*}
 	Here we have used the fact that the Leray projector (the map $\phi\mapsto \eta$) is bounded in $H^s$ for any $s\ge 0$ (see, for example, Chapter 2 and 3 of \cite{lions1997mathematical} or Chapter 2 of \cite{RRS}) and that $H^2(D_+)\subset L^\infty(D_+)$. Since $u(t)$ is weakly incompressible for every $t\in [0,T]$,
 	 we have
 	\begin{equation*}
 	\langle u(t)-u(s),\psi \rangle_{D_+}= \langle u(t)-u(s),\eta+\nabla\sigma \rangle_{D_+}=\langle u(t)-u(s),\eta \rangle_{D_+}.
 	\end{equation*}
 	Since $\eta\in \mathcal{S}_{n,\sigma}(D_+)$ and $\partial _t\eta=0$ it follows from the definition of a weak solution at times $t$ and $s$ that
 	\begin{equation*}
 	\langle u(t)-u(s),\psi \rangle_{D_+} =\int^t_s\langle u(\tau)\otimes u(\tau):\nabla \eta \rangle_{D_+} \d \tau
 	\end{equation*}
 	and hence
 	\begin{equation}\label{phew}
 	|\langle u(t)-u(s),\psi \rangle_{D_+}|\le \|u\|^2_{L^\infty(0,T;L^2(D_+))}\|\nabla \eta\|_{L^\infty(D_+)}|t-s|,
 	\end{equation}
 	which gives \eqref{Lipschitzintime}.
  	Note that as the support of $u$ is $D_+$ we have
 	\begin{equation}
 	|\langle u(t)-u(s),\psi\rangle_{D}|\le C|t-s|\quad \text{for} \, \text{all} \quad \psi\in \mathcal{S}(D),
 	\end{equation}
 	concluding the proof.
 \end{proof}

%
%
A striking corollary of this weak continuity in time is that a mollification of the velocity field  \emph{in space alone} yields a function that is Lipschitz continuous \emph{in time}.

\begin{corollary}\label{Testfunction}
If $u$ is a weak solution on $D_+$ then
	for any $\eps>0$ the functions $J_\eps (u_{E})(x,\cdot)$ and $J_\eps J_\eps (u_{E})(x,\cdot)$ are Lipschitz continuous in $t$ as a function into $L^2(D_+)$:
	\begin{equation}\label{LipschitzL2J}
	\|J_\eps(u_{E})(\cdot,t)-J_\eps (u_{E})(\cdot,s)\|_{L^2(D_+)}\le C_\eps|t-s|,
	\end{equation}
	and
	\begin{equation}\label{LipschitzL2JJ}
	\|J_\eps J_\eps(u_{E})(\cdot,t)-J_\eps J_\eps(u_{E})(\cdot,s)\|_{L^2(D_+)}\le C_\eps|t-s|.
	\end{equation}
	Furthermore, $J_\eps(u_{E}),J_\eps J_\eps(u_{E})\in \mathcal{L}_{n,\sigma}$.
\end{corollary}

\begin{proof}
Set $v= u_E(t)-u_E(s)$; we have the following bounds for the the left-hand sides of \eqref{LipschitzL2J} and \eqref{LipschitzL2JJ}
\begin{align*}
 \| J_\eps v\|_{L^2(D_+)}\le & \| J_\eps v\|_{L^2(D)},\\
 \|J_\eps J_\eps v\|_{L^2(D_+)}\le & \|J_\eps J_\eps v\|_{L^2(D)}\le  \| J_\eps v\|_{L^2(D)}.
\end{align*}
To estimate the right-hand side
$$
 \| J_\eps v\|_{L^2(D)} = \| J_\eps ([u(t)-u(s)] +[u_R(t)-u_R(s)])\|_{L^2(D)}\le 2\| J_\eps ([u(t)-u(s)])\|_{L^2(D)}.
$$
We use the generalisation of Lemma  \ref{uLipcitz} for $\psi \in H^3$.
Let $\psi=J_\eps f$ for $f\in L^2(D)$ with $\|f\|_{L^2(D)}=1$. To find a bound for $\| J_\eps ([u(t)-u(s)])\|_{L^2(D)}$ we notice, following \eqref{phew}, that
\begin{align*}
 |\langle J_\eps ( u(t)-u(s)),f \rangle_{D}|=|\langle   u(t)-u(s),J_\eps f \rangle_{D}|&\le\|u\|_{L^\infty(0,T;L^2(D_+))}^2\|\nabla J_\eps f\|_{L^\infty}|t-s|\\
 &\le C\|u\|_{L^\infty(0,T;L^2(D_+))}^2\|\varphi_\eps\|_{W^{3,1}}|t-s|\|f\|_{L^2}.
\end{align*}
We can then take the supremum over $\|f\|_{L^2}=1$ over both sides to finish off the Lipschitz in time bound and obtain \eqref{LipschitzL2J} and \eqref{LipschitzL2JJ}.

We now need to prove that the other properties of the space $\mathcal{L}_{n,\sigma}$ are satisfied by both $J_\eps u_E$ and $J_\eps  J_\eps u_E$. Since mollification commutes with differentiation we see that both $J_\eps u_E$ and $J_\eps J_\eps u_{E}$ are divergence free. 
	Finally, since $u\in L^\infty(0,T; L^2)$, we observe that both $J_\eps u_E$ and  $J_\eps J_\eps u_{E} \in L^\infty(0,T;H^3)$ and
	\begin{equation*}
	\|J_\eps J_\eps u_{E}\|_{L^1(0,T;H^3)} \le T \|J_\eps J_\eps u_{E}\|_{L^\infty(0,T;H^3)}
	\end{equation*}
	as $[0,T]$ is bounded (similary for $J_\eps u_E$).
	
	We see from Lemma \ref{PropertiesofE} that $J_\eps J_\eps u_{E}\cdot n$ and $J_\eps u_{E}\cdot n=0$ on $\partial D_+$, and hence both $J_\eps u_E$ and $J_\eps J_\eps u_E$ are in $\mathcal{L}_{n,\sigma}$, as required.
\end{proof}
This section (in particular Corollary \ref{Testfunction}) now allows us to use $J_\eps J_\eps(u_{E})$ as test function in the weak formulation of the Euler equations (\ref{WeaksolutionD+}) and have shown the sufficient regularity of $J_\eps(u_{E})$ needed to manipulate terms in the future.

\section{Energy Conservation: $J_\eps J_\eps u_{E}$ as a test function}

Notice that since $J_\eps J_\eps(u_{E})\in \mathcal{L}_{n,\sigma}$  the following identity is a consequence of Lemma \ref{generalpsi}
\begin{multline*}
 \langle u(t),J_\eps J_\eps(u_{E})(t)\rangle_{D_+}-\langle u(0),J_\eps J_\eps (u_{E})(0)\rangle_{D_+}-\int^t_0\langle u(\tau),\partial_t J_\eps J_\eps(u_{E})(\tau)\rangle_{D_+} \d \tau\\
 =\int^t_0\langle u(\tau)\otimes u(\tau):\nabla J_\eps J_\eps(u_{E})(\tau)\rangle_{D_+}\d \tau.
\end{multline*}
Using that the support of $u$ and $u\otimes u$ is $D_+$  we have  for $v= u$ or $u\otimes u$ that
$$
 \langle v,J_\eps J_\eps(u_{E})(t)\rangle_{D_+}= \langle J_\eps v, J_\eps(u_{E})(t)\rangle_{D_{>-\eps}}=\langle J_\eps v, J_\eps(u_{E})(t)\rangle_{D}.
$$
Therefore
%
\begin{multline}\label{Remainderequation}
 \langle J_\eps (u)(t),J_\eps(u_{E})(t)\rangle_{D}-\langle J_\eps (u)(0),J_\eps (u_{E})(0)\rangle_{D}-\int^t_0\langle J_\eps (u)(\tau),\partial_t J_\eps(u_{E})(\tau)\rangle_{D} \d \tau\\
 =\int^t_0\langle J_\eps(u(\tau)\otimes u(\tau)):\nabla J_\eps(u_{E}))(\tau)\rangle_{D}\d \tau.
\end{multline}


We will now investigate the convergence of \eqref{Remainderequation} as $\eps$ tends to zero, and from there deduce energy conservation.

\subsection{Convergence of the L.H.S. of \eqref{Remainderequation}}

In this subsection we want to take limits as $\eps \to 0$ in \eqref{Remainderequation} and  show that the left-hand side becomes
$$
 \frac{1}{2}\left(\|u(t)\|^2_{L^2(D_+)}-\|u(0)\|^2_{L^2(D_+)}\right).
$$
Thus if we show the R.H.S. converges to zero we will have energy conservation. Here we will use the Lipchitz in time regularity of $J_\eps u_E$ shown in Corollary \ref{Testfunction} to manipulate the term with time derivative in the L.H.S. of \eqref{Remainderequation}.

Now, using Lemma \ref{Mollifierconvegence} we can deal with the first two terms, obtaining
$$
 \lim_{\eps \to 0}\left(\langle J_\eps (u)(t),J_\eps(u_{E})(t)\rangle_{D}-\langle J_\eps (u)(0),J_\eps (u_{E})(0)\rangle_{D}\right) = \|u(t)\|^2_{L^2(D_+)}-\|u(0)\|^2_{L^2(D_+)}.
$$
The last term on the left-hand side of \eqref{Remainderequation} can be rewritten using linearity as
\begin{equation}\label{amess}
 \int^t_0\langle J_\eps (u)(\tau),\partial_t J_\eps(u_{E})(\tau)\rangle_{D} \d \tau= \int^t_0\langle J_\eps (u)(\tau),\partial_t J_\eps(u)(\tau)\rangle_{D} \d \tau+\int^t_0\langle J_\eps (u)(\tau),\partial_t J_\eps(u_{R})(\tau)\rangle_{D} \d \tau.
\end{equation}
Since $J_\eps(u)\in C^{0,1}([0,T];H_\sigma)$ we obtain
$$
 2\int^t_0\langle J_\eps (u)(\tau),\partial_t J_\eps(u)(\tau)\rangle_{D} \d \tau=\int^t_0\partial_t\langle J_\eps (u)(\tau), J_\eps(u)(\tau)\rangle_{D} \d \tau= \|J_\eps u(t)\|^2_{L^2(D)}-\|J_\eps u(0)\|^2_{L^2(D)},
$$
and taking limits yields
$$
 \lim_{\eps \to 0} \int^t_0\langle J_\eps (u)(\tau),\partial_t J_\eps(u)(\tau)\rangle_{D} \d \tau= \frac{1}{2}(\| u(t)\|^2_{L^2(D_+)}-\| u(0)\|^2_{L^2(D_+)}).
$$

The only term remaining on the right-hand side of \eqref{amess} that needs to be controlled vanishes:
$$
 \lim_{\eps \to 0} \int^t_0\langle J_\eps (u)(\tau),\partial_t J_\eps(u_{R})(\tau)\rangle_{D} \d \tau=0.
$$
 From Lemma \ref{Rsymmetry} we see that
$$
  2 \int^t_0\langle  J_\eps u, \partial_t J_\eps u_R\rangle_D \d \tau =  \int ^t_0\langle \partial_t J_\eps u, J_\eps u_R\rangle_D+\langle  J_\eps u, \partial_t J_\eps u_R\rangle_D \d \tau=\int ^t_0 \partial_t \langle J_\eps u J_\eps u_R\rangle_D\d \tau,
$$
Therefore it suffices to show that
\begin{equation}\label{alimit}
\lim_{\eps \to 0}\int^t_0\partial_t\langle J_\eps u, J_\eps u_R\rangle_D\d \tau =0.
\end{equation}
Since both $J_\eps u$ and $J_\eps u_R$ are elements of $C^{0,1}([0,T];H_\sigma)$ this integral is equal to
$$
\langle J_\eps u(t),J_\eps u_R(t)\rangle-\langle J_\eps u(0),J_\eps u_R(0)\rangle,
$$
and since the supports of $u(t)$ and $u_R(t)$ are disjoint \eqref{alimit} follows.

We have now shown that the left-hand side of \eqref{Remainderequation} converges to
$$
 \dfrac{1}{2}\left(\|u(t)\|^2_{L^2(D_+)}-\|u(0)\|^2_{L^2(D_+)}\right).
$$

\subsection{Convergence of R.H.S. of \eqref{Remainderequation}}

Recall that the right-hand side of \eqref{Remainderequation} is
$$
 \lim_{\eps \to 0}\left(\int^t_0\langle J_\eps (u\otimes u)(\tau):\nabla J_\eps(u_{E})(\tau)\rangle_{D}\d \tau\right)=:\lim_{\eps \to 0} I,
$$
which we rewrite as
$$
 I=\int^t_0\langle J_\eps (u_E\otimes u)(\tau):\nabla J_\eps(u_{E})(\tau)\rangle_{D}\d \tau +\int^t_0\langle J_\eps ((u-u_E)\otimes u)(\tau):\nabla J_\eps(u_{E})(\tau)\rangle_{D}\d \tau.
$$
For the second term we notice that since $u-u_E$ equals $u_R$ almost everywhere the support of $u$ and $u-u_E$ only intersect in a set measure zero set and so $(u-u_E)\otimes u=0$ a.e.; therefore  the second term vanishes.
For the first term we commute the mollification with the product, using an identity that is similar to one used in previous works \citep{eyink1994energy,constantin1994onsager,cheskidov2008energy,RS09,shvydkoy2010lectures}, but which  involves  two different functions in the product rather than the same function twice. We will use the identity
$$
 J_\eps(u_E\otimes u)=r_\eps(u_E,u)-(u_E-J_\eps (u_E))\otimes(u-J_\eps(u)) + J_\eps u_E \otimes J_\eps u,
$$
with
$$
 r_\eps(u_E,u ):=\int_{D}\varphi_\eps(y)(u_E(x-y)-u_E(x)) \otimes (u(x-y)-u(x))\d y.
$$

Therefore we obtain
$$
 I=\int^t_0\langle [r_\eps(u_E,u)-(u_E-J_\eps (u_E))\otimes(u-J_\eps(u)) + J_\eps u_E \otimes J_\eps u]:\nabla J_\eps(u_{E})(\tau)\rangle_{D}\d \tau.
$$

First we consider the term
$$
\int^t_0\langle J_\eps u_E \otimes J_\eps u: \nabla J_\eps(u_{E})(\tau)\rangle_{D}\d \tau.
$$
If we integrate by parts we obtain
$$
-\frac{1}{2}\int^t_0\int_{D} (\nabla \cdot J_\eps u) |J_\eps (u_{E})|^2\d x \d \tau=0
$$
by incompressibility.

%
%

We are now left with the remainder terms
\begin{equation}\label{Spaceremainderterms}
 \int^t_0\langle [r_\eps(u_E,u):\nabla J_\eps(u_E)(\tau)\rangle_{D}\d \tau  - \int^t_0\langle [(u_E-J_\eps (u_E))\otimes(u-J_\eps(u))]:\nabla J_\eps(u_E)(\tau)\rangle_{D}\d \tau.
\end{equation}
As $(\nabla \varphi)_\eps$ is an odd function, its integral is zero so we can rewrite $\nabla J_\eps (u_E)$ as
\begin{equation}\label{expandedgraduE}
 \nabla J_\eps (u_E)=\int_{D} (\nabla \varphi_\eps)(y)\otimes   (u_E(x-y)-u_E(x)) \d y.
\end{equation}

For the first term in \eqref{Spaceremainderterms},  since $r_\eps(u,u_E)$ is supported in $D_{>-\eps}$ we have
\begin{multline*}
 \int^t_0\langle r_\eps(u_E,u ):\nabla J_\eps (u_E)(\tau)\rangle_{D_{>-\eps}}\d \tau  \\=\int^t_0\bigg\langle \int_{D}\varphi_\eps (y)(u_E(x-y)-u_E(x)) \otimes (u(x-y)-u(x))\d y\,:\\
 \int_{D} (\nabla \varphi_\eps)(z)\otimes   (u_E(x-z)-u_E(x)) \d z\bigg\rangle_{D_{>-\eps}}\d \tau.
\end{multline*}
Using the changes of variables $z=\eps \xi$, $y=\eps \eta$ and taking the modulus we obtain
\begin{multline*}
\left| \int^t_0\langle r_\eps(u_E,u ):\nabla J_\eps (u_E)(\tau)\rangle_{D_{>-\eps}}\d \tau\right| \\
 \le\int^t_0\int_{D_{>-\eps}}\bigg\{ \int_{B_1(0)}|\varphi(\eta)||u_E(x-\eps \eta)-u_E(x)| |u(x-\eps \eta)-u(x)|\d \eta\\\int_{B_1(0)} \frac{1}{\eps}|\nabla \varphi(\xi)|   |u_E(x-\eps \xi)-u_E(x)| \d \xi\bigg\}\,\d x\,\d \tau.
\end{multline*}
Then we can use Fuibini's theorem and   H\"older's inequality  to obtain
\begin{multline}\label{Remainder1}
 \left| \int^t_0\langle r_\eps(u_E,u ):\nabla J_\eps (u_E)(\tau)\rangle_{D_{>-\eps}}\d \tau  \right| \\
 \leq\int_{B_1(0)}|\varphi(\eta)|\|u_E(\cdot-\eps \eta)-u_E(\cdot)\|_{L^3(0,t;L^3(D_{>-\eps}))}|\|u(\cdot-\eps \eta)-u(\cdot)\|_{L^3(0,t;L^3(D_{>-\eps}))}\d \eta
 \\\times\ \frac{1}{\eps}\int_{B_1(0)}|(\nabla \varphi)(\xi)|   \|u_E(\cdot-\eps \xi)-u_E(\cdot)\|_{L^3(0,t;L^3(D_{>-\eps}))} \d \xi.
\end{multline}

For the remaining term in \eqref{Spaceremainderterms}, since  $J_\eps(u)$ is supported in $D_{>-\eps}$ we have
\begin{multline*}
 \int^t_0\langle (u_E-J_\eps u_E)\otimes(u-J_\eps u) :\nabla J_\eps(u_E)(\tau)\rangle_{D_{>-\eps}}\d \tau \\
=\int^t_0\int_{D_{>-\eps}}\bigg\{\int_{D}\varphi_\eps(z)(u_E(x-z)-u_E(x))\d z\ \otimes \\\int_{D}\varphi_\eps(y)(u(x-y)-u(x))\d y\bigg\}\,:\,\int_{D} (\nabla \varphi_\eps)(w)\otimes   (u_E(x-w)-u_E(x)) \d w\d x\d \tau,
\end{multline*}
where  we have used \eqref{expandedgraduE} for the $\nabla J_\eps(u_E)$ term.
As before,  with  the changes of variables $z=\eta \xi$, $y=\eps \zeta$, $w=\eps \xi$ we have
\begin{multline}\label{Remainder2}
\left| \int^t_0\langle (u_E-J_\eps u_E)\otimes(u-J_\eps u) :\nabla J_\eps (u_E)(\tau)\rangle_{D_{>-\eps}}\d \tau\right|\\\le \int_{B_1(0)}|\varphi(\eta)|\|u_E(\cdot-\eps \eta)-u_E(\cdot)\|_{L^3(0,t;L^3(D_{>-\eps}))}\d \eta
\int_{B_1(0)}|\varphi(\zeta)|\|u(\cdot-\eps \zeta)-u(\cdot)\|_{L^3(0,t;L^3(D_{>-\eps}))}\d \zeta
\\ \times\ \frac{1}{\eps}\int_{B_1(0)}|(\nabla \varphi)(\xi)|\|u_E(\cdot-\eps \xi)-u_E(\cdot)\|_{L^3(0,t;L^3(D_{>-\eps}))} \d \xi.
\end{multline}


Before stating our main result, therefore providing sufficient conditions to guarantee that \eqref{Remainder1} and \eqref{Remainder2} vanish in the limit, we remark that if $u(t)\in C^0(\t^2\times[0,\delta])$, for some $\delta>0$, then since $\t^2\times[0,\delta])$ is compact it follows that $u(t)$ is uniformly continuous on $\t^2\times[0,\delta]$. In particular, there exists a non-decreasing function $w_t\:[0,\infty)\to[0,\infty)$ with $w_t(0)=0$ that is continuous at zero, such that
$$
|u(x+y,t)-u(x,t)|\le w_t(|y|).
$$

%
%
%
%
%

We are now ready to state our main result.

\begin{theorem}[Energy Conservation]
Let $u$ be a weak solution of the Euler equations in the sense of Definition \ref{defEulersoln}. Assume that $u$ satisfies
 \begin{itemize}
  \item the bulk condition,
\begin{equation}\label{bulkcondition}
 \lim_{|y|\to 0}\frac{1}{|y|}\int^T_0\int_{D_{>|y|}} |u(x+y)-u(x)|^3\d x\d t=0,
\end{equation}
and
\item continuity near the boundary, $u\in L^3(0,T;C(\t^2\times[0,\delta])$ for some $\delta>0$.
 \end{itemize}
  Then $u$ conserves energy on $[0,T]$, i.e.\ $\|u(t)\|=\|u(0)\|$ for all $t\in[0,T]$.
\end{theorem}

\begin{proof}
It suffices to show that both \eqref{Remainder1} and \eqref{Remainder2}  vanish in the limit as $\eps \to 0$.
First we  would like to bring the limit inside the  integrals  over $B_1(0)$ in both \eqref{Remainder1} and \eqref{Remainder2}. We use the Dominated Convergence Theorem. Since $\varphi \in C^{\infty}_c$ we can find trivial bounds for $\varphi$ and $\nabla \varphi$. Notice that we need to deal with  the factor of $1/\eps$, and factors which are $L^3(0,t;L^3(D_{>-\eps}))$ norms of differences of functions involving $u$ or $u_E$.

We first decompose the $L^3(0,t;L^3(D_{>-\eps}))$ norm by splitting the spatial domain into the bulk area and a strip around the boundary. That is we consider the $L^3(0,t;L^3(D_{>\eps}))$ and $L^3(0,t;L^3(\t^2\times(-\eps,\eps)))$ norms.

For the bulk part, notice that when $x\in D_{>\eps}$ and $\eta,\xi,\zeta\in B_1(0)$ then
$$u_E(x - \cdot \eps)-u_E(x)=u(x - \cdot \eps)-u(x)$$
and we can  therefore define the non-negative function
$$
 f(y)=\frac{1}{|y|}\int^t_0\int_{D_{>\eps}} |\mathbb{I}_{((x+y)\in D_+)}u(x+y)-u(x)|^3\d x\d t
$$
to control the corresponding terms in both  \eqref{Remainder1} and \eqref{Remainder2}. Notice that from the bulk condition \eqref{bulkcondition} it follows that $\lim_{|y|\to 0} f(y)=0$ and therefore for any $\eps >0$ that $\sup_{y\in B_0(\eps)}f(y)\le K$ for some $K=K(\eps)$.
%

We assumed that $u\in L^3(0,T; L^\infty(\t^2\times [0,\eps))$ for $\eps$ sufficiently small, and so using continuity at the boundary and that $u\cdot n=0$ on the boundary we know that $u_E\in L^3(0,T; L^\infty(\t^2\times (-\eps,\eps)))$.
Thus in the region $\t^2 \times (-\eps,\eps)$ we can define the non-negative function
$$
 g(y)=\frac{1}{\eps}\int^t_0\int_{\t^2\times (-\eps,\eps)} |u_E(x+\cdot \eps)-u_E(x)|^3\d x\d t\le \frac{C}{\eps}|\t^2|\eps \int^t_0 \sup_{x\in\t^2\times (-2\eps,2\eps)}|u_E(x)|^3\d t
$$
and see that since $u_E\in L^3(0,T; L^\infty(\t^2\times (-\eps,\eps)))$, the function $g$  is also bounded and integrable. Notice that a similar function $g$ can be defined for the terms involving $u$ instead of $u_E$ as the only property we have used is that $u\in L^3(0,T; L^\infty(\t^2\times [0,\eps))$ for $\eps$ sufficiently small.
%

Using the functions above and the Dominated Convergence Theorem we can move the limit inside the integral, reducing  the problem to showing that
$$
 \limsup_{\eps\to 0}\frac{1}{\eps}\|u(\cdot-\eps \eta)-u(\cdot)\|^3_{L^3(0,t;L^3(D_{>-\eps}))}=C
  $$
  and
  $$
  \lim_{\eps\to 0}\frac{1}{\eps}\|u_E(\cdot-\eps \eta)-u_E(\cdot)\|^3_{L^3(0,t;L^3(D_{>-\eps}))}=0.
$$

We proceed as before, by decomposing $D_{>-\eps}$ into $D_{>\eps}$ and $\t^2\times (-\eps,\eps)$. We first prove the result for the bulk when $x\in D_{>\eps}$. As $\eta\in B_1(0)$ both reduce to showing that
$$
 \lim_{\eps\to 0}\frac{1}{\eps}\|u(\cdot-\eps \eta)-u(\cdot)\|^3_{L^3(0,t;L^3(D_{>\eps}))}=0.
$$
With the change of variables $y=\eps \eta$ for $\eta \in B_1(0)$ we have
$$
 \lim_{|y|\to 0}\frac{1}{|y|}\|u(\cdot-y)-u(\cdot)\|^3_{L^3(0,t;L^3(D_{>\eps}))}=0,
$$
where we have used the  bulk condition \eqref{bulkcondition}.

It remains to show that
\begin{equation}\label{L3boundaryconvegenceconditions1}
 \limsup_{\eps\to 0}\frac{1}{\eps}\|u(\cdot-\eps \eta)-u(\cdot)\|^3_{L^3(0,t;L^3(\t^2 \times (-\eps,\eps)))}=C
 \end{equation}
 and
 \begin{equation}\label{L3boundaryconvegenceconditions2}\lim_{\eps\to 0}\frac{1}{\eps}\|u_E(\cdot-\eps \eta)-u_E(\cdot)\|^3_{L^3(0,t;L^3(\t^2\times (-\eps,\eps)))}=0.
\end{equation}

We now use the continuity of $u$ near the boundary. Now, to deal with \eqref{L3boundaryconvegenceconditions2} note that since the boundary values are the same for $u$ and $u_R$ we have  $u_E(\cdot,t) \in C^0(\t^2\times[-\delta,\delta])$. It follows, since $\partial D^+=\t^2\times\{0\}$ is compact, that for each $t\in[0,T]$ there exists a non-decreasing function $w_t\colon [0,\infty )\to [0,\infty )$ with $w_t(0)=0$ and continuous at $0$, such that
\begin{equation}\label{modcon}
|u(x+z,t)-u(x,t)|<w_t(|z|)
\end{equation}
 whenever $x\in\partial D^+$ and $|z|\le\delta$.

For fixed $t$ and $x'\in \{z=0\}$ we can now write 
\begin{align*}
 |u_E(t,x'+z+y)-u_E(t,x'+z)|&\le |u_E(t,x'+z+y)-u_E(t,x')+u_E(t,x')-u_E(t,x'+z)|\\
 &\le w(t,|y+z|)+w(t,|z|)\\
 &\le 2w(t,2|y|)
\end{align*}
and thus
\begin{align*}
 \frac{1}{|y|}\iint_{\t^2}\int_{-|y|}^{|y|}|u_E(t,x+y)-u_E(t,x)|^3\d x_3\d x_2\d x_1&\le C\frac{1}{|y|}\iint_{\t^2}\int_{-|y|}^{|y|}|w(t,2|y|)|^3\d x_3\d x_2\d x_1\\ 
 &\le C \frac{1}{|y|} |\t^2||y||w(t,2|y|)|^3 \to 0
\end{align*}
as $|y|\to 0$ for almost every $t$.

For the first term in  \eqref{L3boundaryconvegenceconditions1} we can use the fact that $u\in L^3(0,T; L^\infty(\t^2\times (-\delta,\delta)))$ and so
\begin{multline*}
\frac{1}{\eps}\|u(\cdot-\eps \eta)-u(\cdot)\|^3_{L^3(0,t;L^3(\t^2\times (-\eps,\eps)))}=\frac{1}{|y|}\|u(\cdot-y)-u(\cdot)\|^3_{L^3(0,t;L^3(\t^2\times (-\eps,\eps)))}\\\le \frac{C}{|y|}|\t^2||y|\|u(\cdot+y)-u(\cdot)\|_{L^3(0,t;L^\infty(\t^2\times (-\eps,\eps)))}=C|\t^2|\|u(\cdot+y)-u(\cdot)\|_{L^3(0,t;L^\infty(\t^2\times (-\eps,\eps)))}\le C,
\end{multline*}
completing the proof.
\end{proof}

Note that the full strength of the assumption that $u\in L^3(0,T;C^0(\t^2\times[0,\delta])$ is not used in the proof. Rather we require that
\begin{itemize}
\item[(i)] $u\in L^3(0,T;L^\infty(\t^2\times(0,\delta)))$,
\item[(ii)] $u$ is defined pointwise within $\t^2\times[0,\delta]$, and
\item[(iii)] $u(\cdot,t)$ is continuous at every $x\in\partial D^+$;
\end{itemize}
properties (ii) and (iii) together yield \eqref{modcon}.

\section{Conclusion}

Assuming the simple bulk condition
\begin{equation*}
\lim_{|y|\to 0}\frac{1}{|y|}\int^T_0\int_{\t^2\times \R_+} |\mathbb{I}_{(x+y\in \t^2\times \R_+)}u(x+y)-u(x)|^3\d x\d t=0,
\end{equation*}
which is similar to the weakest conditions known on $\R^d$ or $\t^d$, and continuity near the boundary we have proved energy conservation of the incompressible Euler equations with a flat boundary of finite area.
As remarked before this  method does not require any treatment of the pressure; it is an interesting open problem whether energy conservation in a general bounded domain can be proved without involving the pressure.

In \cite{RRS1} we show that one can define a notion of `weak solution' for the Euler equations on a bounded domain that generalises the one we use here, in such a way that the pressure does not appear. Any sufficiently smooth weak solution (understood in this sense) has a corresponding pressure so that the pair $(u,p)$ is a solution in the sense required by \cite{bardos2017}. This means that their argument, while relying on the pressure, is applicable to the (perhaps more natural) definition of weak solution in which the pressure plays no role.

We conclude by pointing out that while we have considered an extension to the full domain, it would have been possible to consider an extension to a smaller strip. The key observation relies on noticing that the main results we have used work when applied to a truncation of the reflection, even for sharp truncations. We state the corresponding {\it local} version of Lemma \ref{Rsymmetry} to illustrate this point; this version of the analysis is carried out in full in \cite{Skipper}.

\begin{lemma}\label{Rsymmetryconclusion}
 For any functions $u$ and $v$ on $D$, define $v_r=\mathbb{I}_{D_{>-\gamma}}v_R$ for some $\gamma>0$. Then
 $$
  \langle u,v_r\rangle_{\t^2\times (-\delta,\delta)}=\langle u_r,v \rangle_{\t^2\times (-\delta,\delta)}
$$
for any $0<\delta <\gamma$.
Further,
$$
  J_\eps (f_r)(x)=J_\eps (f)_r(x)
$$
and thus
$$
  \langle J_\eps u, J_\eps v_r\rangle_{\t^2\times (-\delta,\delta)}=\langle J_\eps u_r,J_\eps v \rangle_{\t^2\times (-\delta,\delta)},
$$
provided $0<\delta\le \gamma-\eps$.
\end{lemma}

The fact that one can consider local versions of the results we have used suggests that these ideas could be transferred to more complicated geometries, where extension to the full domain might be otherwise problematic.

\section*{Acknowledgements}

JLR is currently supported by the European Research Council, grant no.\ 616797. JWDS is supported by EPSRC as part of the MASDOC DTC at the University of Warwick, Grant No. EP/HO23364/1.


\begin{thebibliography}{10}
\providecommand{\natexlab}[1]{#1}
\providecommand{\url}[1]{\texttt{#1}}
\expandafter\ifx\csname urlstyle\endcsname\relax
  \providecommand{\doi}[1]{doi: #1}\else
  \providecommand{\doi}{doi: \begingroup \urlstyle{rm}\Url}\fi

%
%

\bibitem[Bardos \& Titi(2018)]{bardos2017} C.~Bardos \& E.~Titi (2018) Onsager's Conjecture for the Incompressible Euler Equations in Bounded Domains.
    \emph{Archiv. Rat. Mech. Anal.} {\bf 228}, 197--207.

\bibitem[Buckmaster et~al.(2016) Buckmaster, De Lellis, Sz\'{e}kelyhidi, \& Vicol]{delellis2017}
T.~Buckmaster, C.~De~Lellis, L.~Sz\'{e}kelyhidi, \& V.~Vicol (2016)
\newblock \emph{Onsager's conjecture for admissible weak solutions}.
\newblock \emph{Comm. Pure Appl. Math.}, to appear.

\bibitem[Cheskidov et~al.(2008)Cheskidov, Constantin, Friedlander, \&
  Shvydkoy]{cheskidov2008energy}
A.~Cheskidov, P.~Constantin, S.~Friedlander, \& R.~Shvydkoy (2008)
\newblock Energy conservation and {O}nsager's conjecture for the {E}uler
  equations.
\newblock \emph{Nonlinearity} {\bf{ 21}}, 1233--1252.
%
\bibitem[Constantin, E, \& Titi(1994)]{constantin1994onsager}
P.~Constantin, W.~E, \& E.~Titi (1994)
\newblock Onsager's conjecture on the energy conservation for solutions of
  {E}uler's equation.
\newblock \emph{Communications in Mathematical Physics} {\bf{165}},
  207--209.
%
%
%
\bibitem[Duchon \& Robert(2000)]{duchon2000inertial}
J.~Duchon \& R.~Robert (2000)
\newblock Inertial energy dissipation for weak solutions of incompressible
  {E}uler and {N}avier-{S}tokes equations.
\newblock \emph{Nonlinearity} {\bf{13}}, 249--255.
%
\bibitem[Eyink(1994)]{eyink1994energy}
G.~Eyink (1994)
\newblock Energy dissipation without viscosity in ideal hydrodynamics {I}.
  {F}ourier analysis and local energy transfer.
\newblock \emph{Physica D: Nonlinear Phenomena} {\bf{78}}, 222--240.
%
%

\bibitem[Isett(2018)]{isett2016}
P.~Isett (2018)
\newblock A {P}roof of {O}nsager's {C}ojecture.
\newblock \emph{Annals of Math.}, to appear.

%
%
 \bibitem[Lions(1997)]{lions1997mathematical}
 P.-L.~Lions (1997)
 \newblock \emph{Mathematical topics in fluid mechanics, volume {1:} {I}ncompressible
   models.}
 \newblock Oxford University Press.
%
\bibitem[Onsager(1949)]{onsager1949statistical}
L.~Onsager (1949)
\newblock Statistical hydrodynamics.
\newblock \emph{Il Nuovo Cimento (1943-1954)} {\bf{6}}, 279--287.

\bibitem[Robinson, Rodrigo, \& Sadowksi(2016)]{RRS}
J.C. Robinson, J.L. Rodrigo, \& W. Sadowski (2016)
\newblock \emph{The three-dimensional Navier--Stokes equations}.
\newblock Cambridge University Press.

\bibitem[Robinson et~al.(2018) Robinson, Rodrigo, \& Skipper]{RRS1}
J.C.~Robinson, J.L.~Rodrigo, \& J.W.D.~Skipper (2018)
\newblock \emph{Energy conservation in the 3D Euler equations on $\t^2 \times \R_+$}.
\newblock In C.L.\ Fefferman, J.L.\ Rodrigo, \& J.C. Robinson (Eds.), \emph{Partial differential equations in fluid mechanics}, LMS Lecture Notes. Cambridge University Press, Cambridge, UK.

\bibitem[Shvydkoy(2009)]{RS09}
R.~Shvydkoy (2009)
\newblock On the energy of inviscid singular flows.
\newblock \emph{J. Math. Anal. Appl.} {\bf 349},
  583--595.


\bibitem[Shvydkoy(2010)]{shvydkoy2010lectures}
R.~Shvydkoy (2010)
\newblock Lectures on the {O}nsager conjecture.
\newblock \emph{Discrete Contin. Dyn. Syst. Ser. S} {\bf{3}},
  473--496.

\bibitem[Skipper(2018)]{Skipper}
J.W.D.~Skipper (2018)
\newblock \emph{Energy conservation for the Euler equations in domains with boundary}.
PhD Thesis, University of Warwick.








\end{thebibliography}
\end{document}